\documentclass{amsart}[12]
\usepackage{graphicx} 
\usepackage[margin=1in]{geometry}
\usepackage{graphicx}
\usepackage{amsmath, amssymb, amsthm,amsfonts} 
\usepackage{lmodern}
\usepackage[T1]{fontenc}
\usepackage{fancyhdr}
\usepackage[usenames,dvipsnames]{xcolor} 
\usepackage{pdfcolmk}
\usepackage{tikz-cd}
\usepackage{mathtools}
\usepackage{mathrsfs}  
\usepackage{subfiles}
\usepackage{pdfpages}
\usepackage{quiver}

\usepackage{ytableau}

\usepackage{Style}
\usepackage{newclude}

\usepackage[maxbibnames=99,backend=biber, style=alphabetic]{biblatex}
\bibliography{RankSSH.bib}
\usepackage[colorlinks = true, 
linkcolor=black,
filecolor=magenta,
citecolor=magenta]{hyperref}

\newsavebox{\pullback}
\sbox\pullback{%
\begin{tikzpicture}%
\draw (0,0) -- (1ex,0ex);%
\draw (1ex,0ex) -- (1ex,1ex);%
\end{tikzpicture}}

\newcommand{\NS}{\operatorname{NS}}

\newcommand{\BG}{\operatorname{B \mathbb{G}}}

\newcommand{\per}{\operatorname{per}}

\newcommand{\Br}{\operatorname{Br}}

\newcommand{\Auteq}{\operatorname{Auteq}}
\newcommand{\Pf}{\operatorname{Pf}}

\title{Minimal Trivializing Isogenies for $\bG_m$-gerbes over abelian varieties and Period Index Problem}
\author{Ruoxi Li}
\address[R. Li]{Department of Mathematics\\
Univeristy of California, Berkeley\\
Berkeley, CA 94720\\
U.S.A.}
\email{ruoxi\_li@berkeley.edu}

\begin{document}

\maketitle
\begin{abstract}
    For an abelian variety $X$ and $\alpha \in \Br(X)$, we propose a new invariance $\Ind_{SH}(\alpha)$ that refines the known period index relations. It is closely related to the geometry of $\cX$, the $\bG_m$-gerbe over $X$ that corresponds to $\alpha$: we study the minimal trivializing isogenies for $\cX$ via its $\mu_n$-lifts and the $1-$twisted semi-homogeneous vector bundles on $\cX$. As an application, we show that the period index conjecture holds true for products of elliptic curves of any dimension.
\end{abstract}

\section{Introduction}

For any smooth projective variety $X$, the period $\per(\alpha)$ is the order of $\alpha$ in the torsion group $\Br(X)$. The index $\ind(\alpha)$ is the greatest common divisor of the set of ranks  of $\alpha$-twisted vector bundles $\cX$, equivalently, it is the greatest common divisor of the set of degrees of Azumaya algebras of class $\alpha$.


The period–index problem asks how the index of a Brauer class $\alpha \in \Br(X)$ is controlled by its period, more precisely, it is to determine an integer $e$ such that $\ind(\alpha)|\per(\alpha)^e$. The (unramified) period-index conjecture says that for any smooth, connected, projective variety $X$ over an algebraically closed field $k$, and any $\alpha \in \Br(X)$, we have \[\ind(\alpha) | \per(\alpha)^{\dim X-1}.\] It is well-known that $\per(\alpha)| \ind(\alpha)$ (see e.g. \cite[Theorem 3]{antieau2015prime}), and that $\per(\alpha)$ and $\ind(\alpha)$ have the same prime factors (see e.g. \cite[Theorem 6]{antieau2015prime}). While this problem has been extensively studied, it remains subtle even for abelian varieties. However, it is a classical result that \[\ind(\alpha) | \per(\alpha)^{\dim{X}},\] which is by \cite{elencwajg1983projective} in the case of $k = \bC$; for arbitrary algebraically closed field, see, for example, \cite{hotchkiss2024period}.

In this paper, let $k$ be an algebraically closed field of characteristic 0, and let $X/k$ be an abelian variety of dimension $g$. Let $\cX$ be the $\bG_m$-gerbe corresponding to $\alpha \in \Br(X)$, and let $\cX_n$ be a $\mu_n$-gerbe corresponding to $f \in H^2(X, \mu_n)$.

From a geometric point of view, a natural way to study a Brauer class on an abelian variety is via isogenies: one seeks an isogeny $\pi: Y \to X$ such that $\pi^* \alpha =0$. The degree of such an isogeny measures how far $\alpha$ is from being trivial. This leads to the following invariant.

\begin{definition}
\label{definition}
    For $\alpha \in \Br(X)$, let $\ind_{SH}(\alpha)$ be the minimum degree of isogenies $\pi: Y \to X$ such that $\pi^*\alpha = 0$.

    For $f \in H^2(X, \mu_n)$, let $\Ann(f)$ be the minimum degree of isogenies $\pi: Y \to X$ such that $\pi^*f = 0$.
\end{definition}

\begin{remark}
    The notation $\Ind_{SH}(\alpha)$ stands for the fact that this integer is the smallest rank of $\alpha$-twisted semi-homogeneous vector bundles, which is a distinguished class of vector bundles on abelian varieties. This is by Theorem \ref{A}, \ref{B}, and \ref{C}.

    The advantage of considering the invariant $\Ann(f)$ is that it admits a concrete description in terms of alternating pairings on torsion subgroups, making it more accessible to explicit computation. 

    The corresponding notions of $\Ind_{SH}(\alpha)$ and $\Ann(f)$ in the complex tori case are studied in \cite{hotchkiss2025period} ($\Ind_{SH}(\alpha)$ is named as $\Ann(\alpha)$ there). In this case since $\NS(X) = 0$, we have $\Ind_{SH}(\alpha) = \Ann(f)$ where $f$ is the unique lift of $\alpha$ in $H^2(X, \mu_n)$ (assuming $\per(\alpha) | n$).
\end{remark}

For any integer $n$, the Kummer sequence gives the following short exact sequence

\[ 0 \to \frac{\NS(X)}{n \NS(X)} \to H^2(X, \mu_n) \to \Br(X)[n] \to 0.\]
By functoriality, the isogenies that kill any lift of $\alpha$ in $H^2(X, \mu_n)$ also kills $\alpha$, so $\Ind_{SH}(\alpha)$ is bounded above by the minimum of the set of $\Ann(f)$ where $f$ is taken to be any lift of $\alpha$. In this paper, we prove that this bound is sharp (Theorem \ref{C}), and that $\Ind_{SH}(\alpha)$ refines the known period-index divisibility relation (Theorem \ref{D}). 

As we mentioned before, $\Ind_{SH}(\alpha)$ and $\Ann(f)$ are related to the rank of $\alpha$-twisted semi-homogeneous vector bundles. The geometric reason for this is that $\Ind_{SH}(\alpha)$ and $\Ann(f)$ are precisely the smallest degree of the trivializing isogenies of $\bG_m$-gerbes and $\mu_n$-gerbes defined below, respectively.

\begin{definition}
    We say $\pi: Y \to X$ is a trivializing isogeny of $\cX$ if $\cY := \cX \times_X Y$ is trivial over $Y$. We say a trivializing isogeny $\pi$ is minimal if it does not factor through a smaller degree trivializing isogeny of $\cX$. 

    We say $\pi: Y \to X$ is a trivializing isogeny of $\cX_n$ if $\cY_n := \cX_n \times_X Y$ is trivial over $Y$. We say an trivializing isogeny $\pi$ is minimal if it does not factor through a smaller degree trivializing isogeny of $\cX_n$. 
\end{definition}

\subsection{Minimal Trivializing Isogeny of $\mu_n$-gerbes}

Our main results describe minimal trivializing isogenies of $\mu_n$-gerbes in terms of twisted simple semi-homogeneous vector bundles. 

\begin{definition}

    We define the slope of a 1-twisted vector bundle $\cE$ on $\cX_n$ to be $\delta{(\cE)} = \frac{\det(\cE)^{\otimes n}}{n \rk(\cE)}$.
\end{definition}

\begin{remark}
    This definition agrees with the definition of the slope of twisted sheaves in \cite{krashen2008index}.
\end{remark}

\begin{definition}
    Let $\iota: \cX_n \to \cX$ be the pushout map of $\cX_n$ along $\mu_n \to \bG_m$. Any 1-twisted vector bundle $\cE_n$ on $\cX_n$ satisfies $\cE_n \cong \iota^* \cE$ for some $1$-twisted vector bundle $\cE$ on $\cX$. We say $\cE_n$ is semi-homogeneous if $\cE$ is semi-homogeneous on $\cX$ (c.f. \cite[Definition 1.0.1]{li2025point}).
\end{definition}

\begin{theorem}[Theorem 3.1]
\label{A}
    A 1-twisted vector bundle $\cE$ on $\cX_n$ with $\delta(\cE) = 0$ is simple semi-homogeneous if and only if $\cE \cong \pi_* \cL$ for some minimal trivializing isogeny $\pi: Y \to X$ and (1-twisted) $\cL$ on $Y \times_{X} \cX_n$ with $\delta(\cL) = 0$.
\end{theorem}

Furthermore, we can compute the degree of $\pi$ explicitly.
\begin{theorem}[Theorem \ref{Ann{f}=h(f)}]
\label{B}
    Given any $\pi$ a minimal trivializing isogeny of $\cX_n$, we have \[\deg{\pi} = \Ann(f) = h(f),\] where $h(f)$ is the square root of the size of the image of the associated skew symmetric function $X[n] \to X^{\vee}[n]$.

    In particular, all minimal trivializing isogenies of $\cX_n$ have the same degree.
\end{theorem}

\begin{remark}
    By \cite[Theorem 6.1.4]{li2025point}, the rank of simple semi-homogeneous vector bundles (on $\cX$, and hence on $\cX_n$, see the discussion in Section 2) is a numerical invariant. Theorem \ref{A} and \ref{B} implies that $\deg{\pi} = \Ann(f) = h(f)$ is equal to the rank of simple semi-homogeneous vector bundles of slope 0 on $\cX_n$.  
\end{remark}

\subsection{Minimal Trivializing Isogeny of $\bG_m$-gerbes}

The following theorem allows us to compute $\Ind_{SH}(\alpha)$ in terms of the lifts of $\alpha$ in $H^2(X, \mu_n)$.

\begin{theorem}[Corollary \ref{Ann(alpha)}]
\label{C}
    We have $$\Ind_{SH}(\alpha) = \min \{\Ann(f) | f \in H^2(X, \mu_{\per(\alpha)^{s}}), s \in \bN, \text{and the image of $f$ in $\Br(X)$ is $\alpha$}\}.$$ 

    If we further assume $X$ is principally polarized with Picard rank 1 and $\per(\alpha)$ is a prime number, then $$\Ind_{SH}(\alpha) = \min \{\Ann(f) | f \in H^2(X, \mu_{\per(\alpha)}), \text{and the image of $f$ in $\Br(X)$ is $\alpha$}\}.$$
\end{theorem}

\begin{remark}
    As we mentioned after definition \ref{definition} of $\Ind_{SH}(\alpha)$, the above three theorems show that $\Ind_{SH}(\alpha) = \min \{\rk{\cE} | \cE \text{ is an $\alpha$-twisted semi-homogeneous vector bundle} \}$.
\end{remark}

In contrast to the $\mu_n$-gerbe case, the behavior of minimal trivializing isogenies for $\bG_m$-gerbes is more subtle. Although Theorem \ref{C} expresses $\Ind_{SH}(\alpha)$ in terms of $\mu_n$-lifts, minimal trivializing isogenies need not all have the same degree. In Section 4.2, we construct explicit examples where there exist minimal trivializing isogenies of strictly larger degree than $\Ind_{SH}(\alpha)$.

\subsection{Relation with Period-Index Problem}
As a consequence, one obtains a refinement of the period-index relation.

\begin{theorem}[Corollary \ref{divisibility}]
\label{D}
    Given any $f \in H^2(X, \mu_{\per(\alpha)})$ whose image is $\alpha$ in $\Br(X)$, we have

    \begin{enumerate}
        \item $\per(\alpha) | \ind(\alpha) | \Ann(f) | \per(\alpha)^{g}$;
        \item $\ind(\alpha) | \Ind_{SH}(\alpha) | \per(\alpha)^e$ for some integer $e$;
        \item $\Ind_{SH}(\alpha) \leq \Ann(f)$.
    \end{enumerate}
     If we further assume $\per(\alpha)$ is a prime number, we have \[\per(\alpha) | \ind(\alpha) | \Ind_{SH}(\alpha) | \Ann(f) | \per(\alpha)^{g}.\]
  
\end{theorem}

\begin{remark}
    When $\ind(\alpha) = \Ind_{SH}(\alpha)$, we can use $\Ind_{SH}(\alpha)$ to give a strict bound of the period-index problem, while when $\Ind_{SH}(\alpha) = \per(\alpha)^g$ such strategy fails to prove the period-index conjecture. We see in Section 4 that both situations occur. Moreover, $\ind(\alpha)$, $\Ind_{SH}(\alpha)$, and $\per(\alpha)^g$ can be different, making $\Ind_{SH}(\alpha)$ a genuine refinement in the period-index relation.
\end{remark}

The theory of $\Ind_{SH}(\alpha)$ and $\Ann(f)$ has some nice applications. For example, in section 5, we show that any finite product of elliptic curves satisfies the period-index conjecture.

\begin{theorem}[Theorem \ref{Elliptic}]
    Assume $X = E_1 \times \cdots E_g$. Given any $\alpha \in \Br(X)$, let $n = \per(\alpha)$, there exists a lift $f \in H^2(X, \mu_n) \cong H^2(X, \bZ/{n\bZ})$ of $\alpha$ such that $\Ann(f) | n^{g-1}$. In particular, $\ind(\alpha) | n^{g-1}$.
\end{theorem}

\begin{remark}
    One might hope to be able to prove the period-index conjecture for arbitrary abelian varieties by computing $\Ind_{SH}(\alpha)$. Unfortunately, it does not work in some examples when the Néron–Severi group is too small; see Example \ref{counterexample} where $\NS(X) \cong \bZ$. 
    
    The fact that this method can always be used for products of elliptic curves is that the Néron–Severi group is sufficiently large in this case, which makes the Pfaffian condition flexible enough to solve; see section 5 for more details. We expect that similar phenomena can occur for other abelian varieties with sufficiently nice NS groups.
\end{remark}

\subsection{Funding}
The author was supported by a departmental fellowship funded by The McBeth Family Fellowship and The Manferdelli Family Fund in Mathematics

\subsection{Acknowledgments}
The author thanks James Hotchkiss, Tyler Lane, and Martin Olsson for many helpful discussions and suggestions. We also thank Nicolas Addington, Benjamin Antieau, and Daniel Huybrechts for comments and for each pointing me to the reference of \cite{elencwajg1983projective}.


\section{1-twisted Semi-homogeneous vector bundles on $\mu_n$-gerbes}



In this section we study some properties about 1-twisted vector bundles on $\mu_n$-gerbes that will be essential to the rest of the paper.

\begin{proposition}
\label{trivial}
    There exists a $1$-twisted line bundle $\cL$ with $\delta(\cL) = 0$ if and only if $\cX_n \cong  X \times B \mu_n$.
\end{proposition}

\begin{proof}
    If $\cX_n \cong X \times B\mu_n$, then $\cL = \cO_X \otimes \chi$ is a 1-twisted line bundle with slope 0. We now show the converse.

    If there exists $\cL$ on $\cX_n$ with slope 0, we know that $(\det{\cL})^{\otimes n} = p^* \cN$ where $p: \cX_n \to X$ for some $\cN \in X^t(k)$. Let $\cM$ be a line bundle on $X$ such that $\cM^{\otimes n} \cong \cN$, we see that $\cL^{\otimes n} \cong p^* \cN \cong p^* \cM^{\otimes n}$. Now let $\cL' := \cL \otimes \cM^{-1}$, then $\cL'^{\otimes n} \cong \cO_X$, which means that $\cL'$ defines a morphism $\cX_n \to X \times B\mu_n$ such that the pullback of $\cO_X \otimes \chi$ is $\cL'$. Since $\cL'$ is 1-twisted, we see that this morphism is an isomorphism of $\mu_n$-gerbes.
\end{proof}

\begin{remark}
    By \cite[Lemma 2.3.4.2]{lieblich2007moduli}, the existence of 1-twisted line bundle is equivalent to the fact that $\cX_n$ is essentially trivial (c.f. \cite[Definition 2.3.4.1]{lieblich2007moduli}.) The theorem above shows that the existence of 1-twisted slope $0$ line bundle is equivalent to the fact that $\cX_n$ is trivial. In fact, we have the following equivalences.
\end{remark}

\begin{proposition}
\label{trivial mu_n}
    The following are equivalent:
    \begin{enumerate}
        \item $\cX_n \cong X \times B \mu_n$;
        \item there exists a 1-twisted line bundle $\cL$ on $\cX_n$ such that $\delta(\cL) = 0$;
        \item there exists a 1-twisted line bundle $\cL$ on $\cX_n$ such that $\delta(\cL) \in \NS(X)$.
    \end{enumerate}
\end{proposition}

\begin{proof}
    The equivalence of (a) and (b) is Proposition \ref{trivial}, and the equivalence of (b) and (c) can be obtained by tensoring with 0-twisted line bundles.
\end{proof}

Due to Theorem \ref{A}, the notion of semi-homogeneous bundle plays an important role in our study. Here we recall the definition of semi-homogeneous vector bundles.

\begin{definition}
    Let $\iota: \cX_n \to \cX$ be the pushout map of $\cX_n$ along $\mu_n \to \bG_m$. Any 1-twisted vector bundle $\cE_n$ on $\cX_n$ satisfies $\cE_n \cong \iota^* \cE$ for some $1$-twisted vector bundle $\cE$ on $\cX$. We say $\cE_n$ is semi-homogeneous if $\cE$ is semi-homogeneous on $\cX$.
\end{definition}

By the definition of semi-homogeneous vector bundle above and in \cite[Definition 1.0.1]{li2025point}, we have the following commutative diagram

\[
\begin{tikzcd}[column sep=large, row sep=large]
\operatorname{Coh}(\mathcal{X})^{(1)}
  \arrow[r, "\iota^*"]
&
\operatorname{Coh}(\mathcal{X}_n)^{(1)}
\\
\left\{
\begin{array}{c}
\text{1-twisted} \\
\text{semi-homogeneous} \\
\text{vector bundles} \\
\text{on } \mathcal{X}
\end{array}
\right\}
  \arrow[u, hook]
  \arrow[r, "\iota^*"']
&
\left\{
\begin{array}{c}
\text{1-twisted} \\
\text{semi-homogeneous} \\
\text{vector bundles} \\
\text{on } \mathcal{X}_n
\end{array}
\right\}
  \arrow[u, hook]
\end{tikzcd}
\]
where the two horizontal arrows are equivalences.

The concept of (untwisted) semi-homogeneous vector bundle is studied in \cite{mukai1978semi}, and the generalization to the twisted case is studied in \cite{li2025point}, where it is defined as a property of 1-twisted vector bundles on $\bG_m$-gerbes. By the diagram above, we see that the results about 1-twisted semi-homogeneous vector bundles still hold true. We restate some of them here.

\begin{proposition}[Proposition 3.1.3 of \cite{li2025point}]
Let $\pi: Y \to X$ be an isogeny, and also denote $\pi: \cY_n:= \cX_n \times_X Y \to \cX_n$. Let $\cF$ be a 1-twisted vector bundle on $\cY_n$ and $\cE$ a 1-twisted vector bundle on $\cX_n$.
    \begin{itemize}
        \item $\cF$ is semi-homogeneous if and only if $\pi_* \cF$ is semi-homogeneous;
        \item $\cE$ is semi-homogeneous if and only if $\pi^* \cE$ is semi-homogeneous.
    \end{itemize}
\end{proposition}

\begin{proposition}[Theorem 1.1.1 of \cite{li2025point}]
    If $\cE$ is a 1-twisted simple semi-homogeneous vector bundle, then there exists an isogeny $\pi: Y \to X$ and a 1-twisted line bundle $\cL$ on $\cY_n$ such that $\cE \cong \pi_* \cL$.
\end{proposition}

\begin{proposition}[Theorem 1.1.2 of \cite{li2025point}]
    A 1-twisted vector bundle $\cE$ on $\cX_n$ is semi-homogeneous if and only if $\cE \cong \bigoplus_{\cE_i}U_{\cX_n, \cE_i}$ where $U_{\cX_n, \cE_i}$ has a filtration whose successive quotients are isomorphic to a fixed simple semi-homogeneous vector bundles $\cE_i$ ($\cE_i$'s are non-isomorphic) such that all $\delta(\cE_i)$ are equal in $NS(X) \otimes \mathbb{Q}$, and only finitely many $U_{\cX_n, \cE_i} \neq 0$.
\end{proposition}





\section{Minimal Trivializing Isogenies of $\mu_n$-gerbes}

\subsection{Correspondence with Simple Semi-homogeneous Vector bundles}
In this subsection, we prove the following theorem that illustrates the correspondence between minimal trivializing isogenies of $\cX_n$ and the simple semi-homogeneous vector bundle on $\cX_n$. More precisely, the degree of minimal trivializing isogenies agrees with the rank of simple semi-homogeneous vector bundles of slope 0 on $\cX_n$. 

\begin{theorem}
\label{rank=ann}
\begin{enumerate}
    \item A 1-twisted vector bundle $\cE$ on $\cX_n$ with $\delta(\cE) = 0$ is simple semi-homogeneous if $\cE \cong \pi_* \cL$ for some minimal trivializing isogeny $\pi: Y \to X$ and (1-twisted) line bundle $\cL$ on $Y \times_{X} \cX_n$ with $\delta(\cL) = 0$;
    \item A trivializing isogeny $\pi: Y \to X$ for $\cX_n$ is minimal if $\cE:= \pi_* \cL$ is simple semi-homogeneous with $\delta(\cE)=0$.
\end{enumerate}
\end{theorem}

\begin{proof}
    By \cite[proposition 3.3]{li2025point}, we know that given any trivializing isogeny $\pi: Y \to X$ for $\cX_n$, $\pi_* \cL$ is semi-homogeneous for any line bundle $\cL$ on $\cY_n := \cX_n \times_X Y$. Since $\pi^* \pi_* \cL \cong \oplus_{a \in \ker{\pi}} a \cdot \cL$, we see that \[[\pi^*\det(\cE)^{\otimes \rk{\cE}}]= [(\det{\oplus_{a \in \ker{\pi}} a \cdot \cL})^{\otimes \rk{\cE}}] \cong [(\det(\cL^{\otimes \deg{\pi}}))^{\otimes \rk{\cE}}] \in \NS(Y)\] and $\pi^*: \NS(X) \to \NS(Y)$ is always injective, so $\delta(\cE) = 0$ if and only if $\delta({\cL}) = 0$. So we need to show that $\pi_* \cL$ is simple if and only if $\pi$ is minimal.

    \[\Hom_{\cX_n}(\pi_* \cL, \pi_* \cL) \cong \Hom_{\cY_n}(\pi^* \pi_* \cL, \cL) \cong \Hom_{\cY_n}(\oplus_{a \in \ker{\pi}}a \cdot \cL, \cL) \]
    where $a \cdot \cL$ denotes the descent action of a on $\cL$.
    
    We see that $\pi_* \cL$ is simple if and only if $a \cdot \cL \not \cong \cL$ for any $a \neq 0$ if and only if $\cL$ does not descend to $\cX_n \times_X Y'$ for any $Y \to Y' \to X$ with $\deg(Y \to Y') >1$, by Proposition \ref{trivial}, this is equivalent to the fact that $\pi$ is minimal since $\cL$ has degree 0. 

    Note that the fact that $\cL$ does not descend to $\cX_n \times_X Y'$ for any $Y \to Y' \to X$ implies $a \cdot \cL \not \cong \cL$ for any $a \neq 0$ can be seen as follows: if there exists some $a \in \ker{\pi}$ such that $a \cdot \cL \cong \cL$, then let $Y' := Y/<a>$, and let $\cY'_n : = \cX_n \times_X Y'$. We have that $\cL$ descend to $\cY'_n$.
\end{proof}

\begin{corollary}
    If $\cE \cong \pi_* \cL$ is a 1-twisted simple semi-homogeneous vector bundle on $\cX_n$ with $\delta(\cE) \in \NS(X)$, then $\pi$ is a minimal trivializing isogeny.
\end{corollary}
\begin{proof}
    Let $[\cN] \in \NS(X)$ be the class of $\delta(\cE)$, take $\cE' := \cE \otimes \cN^{-1}$. Then \[\delta(\cE') = \frac{\det(\cE \otimes \cN^{-1})^{\otimes n}}{n \rk{\cE}} = \frac{\det(\cE)^{\otimes n} \otimes \det(\cN^{-1})^{\otimes n \rk{\cE}}}{n \rk{\cE}} = \delta(\cE) + [\cN^{-1}] = 0.\]

    And we see that \[\cE' \cong \cE \otimes \cN^{-1} \cong \pi_* \cL \otimes \cN^{-1} \cong \pi_*(\cL \otimes \pi^* \cN^{-1}).\]
    So by Theorem \ref{rank=ann}, we see that $\pi$ must be a trivializing isogeny for $\cX_n$.
\end{proof}

\begin{remark}
    If we remove the assumption of $\delta(\cE) = 0$, by \cite[Theorem 1.1.1]{li2025point}, it is still true that $\cE \cong \pi'_* \cL'$ where $\pi': Y' \to X$ is an isogeny and $\cL'$ is a 1-twisted line bundle on $\cY' := \cX \times_X Y'$. However, $\pi'$ is not necessarily a minimal trivializing isogeny, as in the following proposition (although such $\pi'_* \cL'$ cannot have $\delta \in \NS(X)$). Furthermore, $\pi'$ does not even need to be a trivializing isogeny, that is, $\cY'$ does not need to be the trivial $\mu_n$-gerbe over $Y$, it just needs to be an essentially trivial $\mu_n$-gerbe (i.e. the pushout $\bG_m$-gerbe is trivial) so that a 1-twisted line bundle exists on $\cY'$.
\end{remark}


\begin{proposition}
\label{Prop: example}
    Assume $\dim{X} = 1$, and let $r$ be the rank of 1-twisted simple semi-homogeneous vector bundles on $\cX_n$ whose slope is $0$. For any integer $m$, there exists a 1-twisted simple semi-homogeneous vector bundle on $\cX_n$ with rank $mr$.
\end{proposition}

\begin{proof}
    Take any minimal trivializing isogeny $\pi: Y \to X$ of $\cX_n$, let $f: Y' \to Y$ be a degree $m$ isogeny. By Proposition \ref{trivial}$, \cX_n \times_X Y'$ is trivial so there exists a 1-twisted line bundle $\cN$ with $\cN^{\otimes n} \in X^t(k)$ and this $\cN$ defines an isomorphism $\cX_n \times_X Y' \cong X \times B\mu_n$, here we choose the trivialization to be the one corresponding to the pullback of $\cN$, that is, the trivialization over $Y$. Take $\cL$ line bundle on $\cX_n \times_X Y' \cong Y' \times B\mu_n$ such that $\cL \otimes \chi^{-1}$ is a representative of $1 \in NS(X)$. We know that such a line bundle never descends through any nontrivial isogeny for elliptic curves, so $\cL \otimes \chi^{-1}$ does not descend to $\cY_n$. Since the trivializations for the $\mu_n$-gerbes are compatible with $f: Y' \to Y$, we see that $\cL$ cannot descend to $\cY_n$. This then implies that $\cL$ does not descend to any 1-twisted line bundle $H$ on $\cX_n$, since otherwise it descends to $\pi^* H$ on $\cY_n$. Now, for similar reasons as in the proof of Theorem \ref{rank=ann}, we have that $(f \circ \pi)_* \cL$ is 1-twisted simple semi-homogeneous on $\cX_n$ as desired.
\end{proof}

\subsection{Relation with the Value $\Ann(f)$}

In this subsection, we compute the value $\Ann(f)$ by viewing $f$ as an alternating pairing in $H^2(X, \mu_n) \cong \Hom(\wedge X[n], \mu_n)$, and we show that the degree of minimal trivializing isogenies of $\cX_n$ is $\Ann(f)$.

\begin{definition}
    For any $f \in H^2(X, \mu_n)$, $f$ corresponds to a skew-symmetric function $M: X[n] \to X^{\vee}[n]$, we denote $h(f)$ to be the size of the square root of the image of $M$.
\end{definition}

\begin{theorem}
\label{Ann{f}=h(f)}
    Given any $\pi$ a minimal trivializing isogeny of $\cX_n$ which corresponds to $f \in H^2(X, \mu_n)$, we have $\deg{\pi} = \Ann(f) = h(f)$.   
\end{theorem}

\begin{proof}
    By definition of $\Ann(f)$, we see that there is a minimal trivializing isogeny of $\cX_n$ of degree $\Ann(f)$. By Lemma \ref{ann(f)}, we see that for any $f \in H^2(X, \mu_n)$, we have $\Ann(f) = h(f)$. 
    
    So it suffices to show that $\deg{\pi} = \Ann(f)$ for any minimal trivializing isogeny $\pi$ of $\cX_n$. Given any such $\pi$, by Proposition \ref{rank=ann}, $\deg{\pi}$ equals to the rank of some 1-twisted simple semi-homogeneous vector bundle $\cE$ on $\cX_n$ with $\delta(\cE) = 0$. By \cite[Corollary 4.2.5 or Proposition 6.1.4]{li2025point}, all such $\cE$ have the same rank. Therefore $\deg{\pi} = \Ann(f) = h(f)$.
\end{proof}

\begin{remark}
    In particular, this theorem shows that all minimal trivializing isogenies of $\cX_n$ have the same degree.
\end{remark}

We now prove the essential lemmas we used to prove the above theorem.

Let 
\[
f : X[n] \times X[n] \to \mu_n
\]
be an alternating bilinear pairing. Let
\[
R := \{ x \in X[n] \mid f(x,y) = 1 \text{ for all } y \in X[n] \}.
\]
Then \( f \) descends to
\[
\overline{X} := \frac{X[n]}{R},
\]
i.e.
\[
\overline{f} : \overline{X} \times \overline{X} \to \mu_n, 
\quad \overline{f}(\overline{x}, \overline{y}) = f(x,y),
\]
and \( \overline{f} \) is nondegenerate.

In particular, isotropic subgroups descend to isotropic subgroups.

\medskip

\begin{lemma}
    Maximal isotropic subgroups of \( f \) have size
\[\frac{|X[n]|}{\sqrt{|\im M|}}
\]
\end{lemma} 

\medskip

\begin{proof}
    Let \( G \subset X[n] \) be a maximal isotropic subgroup with respect to \( f \), i.e. \( G = G^\perp \).
Then its image \( \overline{G} \subset \overline{X} = X[n]/R \) is also maximal isotropic, since \( \overline{f} \) is nondegenerate. Hence
\[
|\overline{G}| = \sqrt{|\overline{X}|} = \sqrt{\left| \frac{X[n]}{R} \right|}.
\]
Thus
\[
|G| = |\overline{G}| \cdot |R| = \sqrt{|R| \cdot |X[n]|}.
\]
By construction, $R = \ker M,$ so
\[
|R| \cdot |\operatorname{im} M| = |X[n]|.
\]
Therefore
\[
|G| = \frac{|X[n]|}{\sqrt{|\operatorname{im} M|}}.
\]
\end{proof} 

\begin{lemma}
\label{ann(f)}
    $\Ann(f) = \sqrt{|\im M|}$.
\end{lemma}

\begin{proof}
We first show that $$\operatorname{Ann}(f) \le \sqrt{|\operatorname{im} M|}.$$

Let $f : X[n] \times X[n] \to \mu_n,$
and let $G \subset X[n]$ be a maximal isotropic subgroup. Let $Y := X/G.$
Since $G \subset X[n]$, there exists an isogeny $\pi : Y \to X$ such that $[n]_X = \pi \circ q$ for the quotient $q : X \to Y:= X/G$.

We claim that $\pi(Y[n]) = G.$ Indeed, for $y \in Y$, take some $x \in X$ such that $q(x) = y$, we have \[y \in Y[n] \iff ny = 0 \iff q(nx) = 0 \iff q(\pi(q(x)))=0 \iff q(\pi(y))=0,\] so $ \pi(y) \in G $. 


Now we compute: $\pi^* f (y_1,y_2) = f(\pi y_1, \pi y_2) = 1$, since $ \pi y_1, \pi y_2 \in G $ and $G$ is isotropic. Hence $\pi^* f = 0 \in H^2(Y, \mu_n)$.
By the previous lemma, $|G| = \frac{|X[n]|}{\sqrt{|\operatorname{im} M|}}.$ Therefore,
\[
\deg \pi = \frac{\deg [n]}{\deg q} = \frac{|X[n]|}{|G|} = \sqrt{|\operatorname{im} M|}.
\]
Thus $\operatorname{Ann}(f) \le \deg \pi = \sqrt{|\operatorname{im} M|}.$

Now we show $\operatorname{Ann}(f) \ge \sqrt{|\operatorname{im} M|}$.
Let $ \pi : Y \to X $ be any isogeny such that $$\pi^* f = 0 \in H^2(Y, \mu_n).$$ Then $S := \pi(Y[n])$ is an isotropic subgroup of $ X[n] $. Moreover, $\ker \pi \cap Y[n] = \ker(\pi|_{Y[n]})$, so 
\[
|\ker \pi| \ge |\ker(\pi|_{Y[n]})|
= \frac{|Y[n]|}{|\operatorname{im}(\pi|_{Y[n]})|}
= \frac{|Y[n]|}{|S|}.
\]

Since $ |Y[n]| = |X[n]| $, we get $|\ker \pi| \ge \frac{|X[n]|}{|S|}.$ Since $S$ is isotropic, $|S| \le \frac{|X[n]|}{\sqrt{|\operatorname{im} M|}},$ hence $|\ker \pi| \ge \sqrt{|\operatorname{im} M|}.$ Therefore, $\operatorname{Ann}(f) \ge \sqrt{|\operatorname{im} M|}.$

Combining both inequalities, we conclude
\[
\operatorname{Ann}(f) = \sqrt{|\operatorname{im} M|}.
\]

\end{proof}

\section{Trivializing Isogenies of $\bG_m$-gerbes}

\subsection{The Quantity $\Ind_{SH}(\alpha)$}

In this subsection, we show that $\Ind_{SH}(\alpha)$ is controlled by its lifts in $H^2(X, \mu_{\per{\alpha}})$. For more precise relation, see Corollary \ref{Ann(alpha)}.


\begin{proposition}
\label{trivializing lifts}
    Given any trivializing isogeny of $\alpha$, $\pi: Y \to X$ of degree $r$, there exists some $f \in H^2(X, \mu_r)$ that is a lift of $\alpha$ such that $\pi^* f = 0$.
\end{proposition}

\begin{remark}
\label{diagram}
    Under the assumptions in the proposition, we have
\[
\begin{tikzcd}
0 \arrow[r] 
& \dfrac{\NS(X)}{r\NS(X)} \arrow[r] \arrow[d, "\pi_1"] 
& H^2(X,\mu_r) \arrow[r] \arrow[d, "\pi_2"] 
& \operatorname{Br}(X)[r] \arrow[r] \arrow[d, "\pi_3"] 
& 0 \\
0 \arrow[r] 
& \dfrac{\NS(Y)}{r\NS(Y)} \arrow[r] 
& H^2(Y,\mu_r) \arrow[r] 
& \operatorname{Br}(Y)[r] \arrow[r] 
& 0
\end{tikzcd}
\]
By the snake lemma, we have exact sequence
\[\ker(\pi_2) \rightarrow \ker(\pi_3) \xrightarrow{\partial} \operatorname{coker}(\pi_1).\]
\end{remark}

\begin{proof}
    We see that it suffices to show $\partial(\alpha)= 0$. Since $P$ is a lift of $\alpha$ and $\pi^* P \cong \bP E$, chasing through the diagram, we get $\partial(\alpha) = [\det E] \in \frac{\NS(Y)}{r\NS(Y)} \subset H^2(Y, \mu_r)$. Now we show $[\det E] = 0 \in \coker{\pi_1} \cong \frac{\NS(Y)}{r\NS(Y) + \pi^* \NS(X)}$.

    Since $\pi^* \alpha = 0$, by \cite[Proposition 3.1.3]{li2025point} $\pi_* \cL$ for any 1-twisted line bundle $\cL$ on $\cY := \cX \times_X Y \cong Y \times \BG_m$ is a semi-homogeneous vector bundle of rank $r$ on $\cX$. $\pi_* \cL$ correspond to an element $P \in H^1(X, \PGL_r)$ whose image is $\alpha$ in $\Br(X)[r]$. By the construction of $P$, we know that $\pi^* P \cong \bP E$ for $E \cong \oplus_{a \in \ker{\pi}}{a \cdot \cL} \otimes \chi^{-1}$ where $\chi$ here is a chosen trivialization of $ \cY \cong Y \times \BG_m$, under this trivialization, $\cL \cong L \otimes \chi^{-1}$ for some $L$ on Y. So $\det{E} = \otimes_{a \in \ker{\pi}} ({a \cdot \cL} \otimes \chi^{-1})$ where $a \cdot (-)$ is the descent action. By \cite[8.11]{olsson2025twisted}, $a \cdot \cL \otimes \chi^{-1} \cong t_a^* L \otimes \cN_a$ for some degree $0$, 0-twisted line bundle $\cN_a$, so $\det{E}$ is equivalent to ${(\cL \otimes \chi^{-1})}^{\otimes r}  \in r \NS(Y) \subset \NS(Y)$. We then conclude that $\partial(\alpha) = 0$.


\end{proof}

From now on we fix compatible identification of $\mu_n \cong \bZ / n \bZ$ for every $n$. Denote $$\Ann_n(\alpha) := \min \{\Ann(f), f \in H^2(X, \mu_n) \text{ whose image is $\alpha \in \Br(X)$} \}.$$

By the natrual inclusion $\bZ / n \bZ \to \bZ / tn \bZ$, we have the map \[H^2(X, \bZ / n \bZ) \xrightarrow{M \mapsto tM} H^2(X, \bZ/ tn \bZ).\] 

For $M$ a skew symmetric matrix in $H^2(X, \bZ / n \bZ)$, this map sends it to $tM \in H^2(X, \bZ / tn \bZ)$: entry-wise, $tM$ is gotten by multiplying each entry of $M$ by $t$ viewed as elements in $\bZ / tn \bZ$.

\begin{lemma}
\label{basic inequalities}
    \begin{enumerate}
        \item $\Ann_{tn}(\alpha) \leq \Ann_n(\alpha)$;
        \item suppose $\gcd(t,n) = 1$, then $\Ann_{n}(\alpha) \leq \Ann_{tn}(\alpha)$;
        \item suppose $\NS(X) \cong \bZ \cdot H$ and $H$ is principally polarized, and assume the prime factors of $t$ are contained in the prime factors of $n$, then $\Ann_n(\alpha) \leq \Ann_{tn}(\alpha)$.
    \end{enumerate}
\end{lemma}

\begin{proof}
    (a) follows from the fact that $|\im{M}| = |\im{tM}|$ for every $M$ in $H^2(X, \bZ/n \bZ)$.

    We then show (b). Let $H_i$ denote the generators of $\NS(X)$, we abuse the notation and let $M_{H_i}$ denote the corresponding matrices in  $H^2(X, \bZ / n \bZ)$, $H^2(X, \bZ/ t \bZ)$, and $H^2(X, \bZ / tn \bZ)$.

    To show the statement it suffices to show that for any linear combination $\sum a_i H_i$, $|\im{tM + \sum a_i H_i}| \geq |\im {M + \sum a_{i,n} H_i}|$, where $a_{i,n}$ is the reduction of $a_i$ in $\bZ / n \bZ$.

    By Chinese Remainder Theorem, we have that $H^2(X, \bZ / tn \bZ) \cong H^2(X, \bZ/ n \bZ) \times H^2(X, \bZ / t \bZ)$, and via this isomorphism $tM + \sum a_i H_i$ gets sent to $(tM + \sum a_{i,n}H_i, \sum a_{i, t}H_i)$.

    So \[|\im{tM + \sum a_i H_i}| = |\im tM + \sum a_{i,n}H_i| \cdot |\im \sum a_{i, t}H_i| \geq |\im tM + \sum a_{i,n}H_i| = |\im M + \sum t^{-1}a_{i,n}H_i|\] as desired, here $t^{-1}$ is taken mod $n$.

    Now we show (c) in the following steps. Let $m=tn$, and let $M_H$ denote the corresponding matrices in $H^2(X, \bZ / n\bZ)$ and $H^2(X, \bZ / m \bZ)$. Since $H$ is principally polarized, we may choose the basis of $H^1(X, \bZ)$ so that $M_H = 
    \mathrm{diag}\!\left(
\begin{pmatrix}0 & 1 \\ -1 & 0\end{pmatrix},
\ldots,
\begin{pmatrix}0 & 1 \\ -1 & 0\end{pmatrix},
\ldots,
\begin{pmatrix}0 & 1 \\ -1 & 0\end{pmatrix}
\right).$
    
    Step 1. We first show the statement for $n = p^u$, $m = p^s$ where $u < s \in \bN$.

    Given any $f' \in H^2(X, \bZ / m \bZ)$ whose image in $\Br(X)$ is $\alpha$, $f'$ corresponds to a alternating matrix $M_{f'} : X[p^s] \to X^t[p^s]$, and there exists a class $[aH]$ with $a < p^s$ such that \[M_{f'} = p^{s-u} M_{f} + M_{aH}\] for some $f \in H^2(X, \bZ / n \bZ)$ whose image is $\alpha$.

    Case (1). $p \nmid a$, then $\det{M_{f'} = \det{a M_{H}}} = a^{2g} =  1$ mod $p$, which means $\det{M_{f'}}$ is invertible mod $p^s$. So $M_{f'}$ is invertible and hence $|\im M_{f'}| = p^{2sg} > p^{2g}$, which is the upper bound for $|\im M_{f}|$.

    Case (2). $p \mid a$ but $p^{s-u} \nmid a$. We let $v$ be the greatest integer such that $p^v \mid a$, then $v < s-u$ and $a = rp^v$ for some $|r| < p$.

    So we have $M_{f'} = p^v(p^{(s-v-u)} M_f + r M_{H})$. Similar as in case 1, we see that $p^{(s-v-u)} M_f + r M_{H}$ is invertible mod $n$, which implies that $\im{M_{f'}} = \im{p^v I} = p^{2g(s-v)} > p^{2g}$. So $|\im M_{f'}| \geq |\im M_{f}|$ since $v < s$.

    Case (3). $p^{s-u} \mid a$, then $a = rp^{s-u}$ for some $|r| < p$. So $M_{f'} = p^{s-u}(M_f + M_{rH})$ which then is in the image of $H^2(X, \bZ / n \bZ) \to H^2(X, \bZ / m \bZ)$. So $|\im M_{f'}| = | \im (M_{f} + M_{rH})|$. Note that $f + rH \in H^2(X, \bZ / n \bZ)$ also has image $\alpha \in \Br(X)$.

    Step 2. We now show the statement for $n = p_1^{e_1} \cdots p_k^{e_k}$, $m = p_1^{d_1} \cdots p_k^{d_k}$ where $d_i \geq e_i$.

    Similar as in step 1, given any $f' \in H^2(X, \bZ / m \bZ)$ whose image is $\alpha$, we see that \[M_{f'} = p_1^{d_1 - e_1} \cdots p_k^{d_k - e_k} M_f + a M_H\] for some $f \in H^2(X, \bZ / n \bZ)$ whose image is $\alpha$.

    By Chinese Remainder Theorem, We may view \[M_{f} = (M_{f, p_1^{e_1}}, \cdots, M_{f, p_k^{e_k}}) \in H^2(X, \bZ / p_1^{e_1} \bZ) \times \cdots \times H^2(X, \bZ / p_k^{e_k} \bZ);\] and
    \[M_{f'} = (p_1^{d_1 - e_1}M_{f, p_1^{e_1}} + a_{p_1^{d_1}} M_H, \cdots, p_k^{d_k - e_k} M_{f, p_k^{e_k}} + a_{p_k^{d_k}} M_H) \in H^2(X, \bZ / p_1^{d_1} \bZ) \times \cdots \times H^2(X, \bZ / p_k^{d_k} \bZ)\] where $a_{p_i^{d_i}}$ denotes the reduction of $a$ in $\bZ / p_i^{d_i} \bZ$.

    By step 1, we see that the fact that $|\im M_{f'} + aH|$ is the smallest implies that $p_i^{d_i - e_i} \mid a_{p_i^{d_i}}$, which is equivalent to the statement that $a = b\frac{m}{n}$ for some integer $b$. 
    
    So we see that $\Ann_m(\alpha) = |\im {M_{f'} + a H}| = |\im {M_f} + b H| \geq \Ann_n(\alpha)$ as desired.
\end{proof}

\begin{corollary}
\label{Ann(alpha)}
    Given any $\alpha \in \Br(X)$ with $\per(\alpha) = n$, we have $\Ind_{SH}(\alpha) = \min_{s \in \bN} \{\Ann_{n^s}(\alpha)\}$. 
    
    If we further assume $X$ is principally polarized with Picard rank 1 and $n$ is a prime number, then $\Ind_{SH}(\alpha) = \min \{\Ann_n(\alpha)\}$.
\end{corollary}
\begin{proof}
    By Proposition \ref{trivializing lifts} and (a), (b) of Lemma \ref{basic inequalities}, we see that $\Ind_{SH}(\alpha) = \min \{\Ann(f)\}$ where $f$ ranges over lifts of $\alpha$ in $H^2(X, \mu_{tn})$ for $t \in \bZ$ with same prime factors as $n$. For any $t$, we can find an $l$ such that $tn | n^s$. We then conclude the first statement.

    The second statement then follows from (c) of Lemma \ref{basic inequalities}.
\end{proof}

\begin{corollary}
\label{divisibility}
    Given any $f \in H^2(X, \mu_{\per(\alpha)})$ whose image is $\alpha$ in $\Br(X)$, we have

    \begin{enumerate}
        \item $\per(\alpha) | \ind(\alpha) | \Ann(f) | \per(\alpha)^{g}$;
        \item $\ind(\alpha) | \Ind_{SH}(\alpha) | \per(\alpha)^e$ for some integer $e$;
        \item $\Ind_{SH}(\alpha) \leq \Ann(f)$.
    \end{enumerate}
     In particular, if $\per(\alpha)$ is a prime number, we have \[\per(\alpha) | \ind(\alpha) | \Ind_{SH}(\alpha) | \Ann(f) | \per(\alpha)^{g}.\]
\end{corollary}

\begin{proof}
    As mentioned before, $f$ corresponds to a skew symmetric function $X[\per(\alpha)] \to X^t[\per(\alpha)]$, so the image of $f$ is a subgroup of $X^t[\per(\alpha)]$, which means $\Ann(f) = h(f) | \per(\alpha)^g$. By definition of $\ind(\alpha)$, since $\Ann(f)$ is the rank of some vector bundle on $\cX$ (Theorem \ref{rank=ann}), we also have that $\ind(\alpha) | \Ann(f)$. This proves (a).

    let $f' \in H^2(X, \mu_{n^s})$ be the element such that $\Ind_{SH}(\alpha) = \Ann(f')$, a similar argument shows (b).

    By Corollary \ref{Ann(alpha)}, we have (c).
\end{proof}





\subsection{Example: Minimal Trivializing Isogenies for $\cX$ with Larger Degree}
Corollary \ref{Ann(alpha)} gives a formula to compute $\Ind_{SH}(\alpha)$, but unlike the $\mu_n$-gerbe case, it is not true that every minimal trivializing isogeny for $\cX$ have the same degree, in particular, there are minimal trivializing isogenies for $\cX$ with degree greater than $\Ind_{SH}(\alpha)$.

In this subsection, we give an example that illustrates this behavior. We first prove the following lemma that is useful for understanding the construction.

\begin{lemma}
\label{gp hom}
    Given an isogeny $\pi: Y \to X$ and some $P \in H^1(X, \PGL_r)$ for some r such that $\pi^* P \cong \bP E$ for some vector bundle $E$ on Y, the descent data of $\bP E$ gives rise to a group homomorphism $\ker{\pi} \to Y \times Y^t$.
\end{lemma}

\begin{proof}
    Given any $a \in \ker{\pi}$, the descent data gives isomorphism $\gamma_a: \bP t_a^* E \cong \bP E$. We then must have $\gamma_a^*\cO_{\bP E}(1) \cong \cO_{\bP t_a^* E} \otimes q^* \cN_a$ where $q: \bP E \to Y$ and $\cN_a$ is some line bundle on Y. Applying $q_*(-)$, we get \[t_a^* E \cong E \otimes \cN_a.\]

    We define our map $\ker{\pi} \to Y \times Y^t$ as sending $a$ to $(a, \cN_a)$. Since descent data satisfies cocycle conditions, one can check that this map is a group (scheme) homomorphism.
\end{proof}

\begin{remark}
     When $E$ is simple semi-homogeneous, the group homomorphism maps into $\Phi(E):= \{ (a, \cN) \in Y \times Y^t| t_a^* E \cong E \otimes \cN \}$, which is a subabelian variety of $Y \times Y^t$. For more information about $\Phi(E)$, see \cite[Section 3 and 7]{mukai1978semi}. 
     
     The reason why this map exists is that $\Aut^0_{\bP E} \subset Y \times Y^t$. In fact, one can show that when $E$ is simple semi-homogeneous, $\Aut^0_{\bP E} \cong \Phi(E)$. This is discussed in more detail in a forthcoming joint work on derived categories of homogeneous projective bundles over abelian varieties with Feiyang Lin.
\end{remark}

\textbf{We now give the example.}

Let $n = p_1 \cdots p_m$ be some integer with no repeated factor (i.e. square-free,) let $A$ be an abelian variety of dimension $g$ such that $\NS(A) \cong \bZ$ with generator having polarization type $(1, 1, , \cdots, 1, n)$ (the first $g-1$ components are 1), then there exists a line bundle on $A$ with $\chi = n$. The usual Fourier Mukai transform $D(A) \cong D(A^t)$ gives a simple semi-homogeneous vector bundle $E$ on $Y := A^t$ with $\rk(E) = n$ and $\chi(E) = 1$. Note also that $Y$ has a natural dual polarization type $(1, n, \cdots, n)$ (the last g-1 components are n,) we see that this means that the generator of $\NS(Y)$ has polarization type $(1, n, \cdots, n)$.

We prove such an abelian variety $A$ exists and the fact that $\rk(E) = n$ and $|\chi(E)| = 1$ at the end of the subsection in Lemma \ref{existance of A}.

Let $\cL := \det{E}$. By \cite[Proposition 6.12]{mukai1978semi}, $|\chi(\cL)| = n^{g-1}$. Since the generator of $\NS(Y)$ has the polarization type $(1, n, \cdots, n)$, we see that $\cL$ must be the the ample generating class (or the inverse of it) and have the polarization type $(1, n, \cdots, n)$.

Let $\varphi: = [n] \circ \phi_{\cL}$, where $\phi_{\cL}: Y \to Y^t$ is the symmetric homomorphism given by $\cL$. We see $\varphi = \phi_{\cL^{\otimes n}}$, which then have \[\ker{\varphi \cong (\bZ/{n\bZ}} \times \bZ/{n^2 \bZ} \times \cdots \times \bZ/{n^2 \bZ})^{\oplus 2}.\] This is because $\cL^{\otimes n}$ has polarization type $(n, n^2, \cdots, n^2)$.

Now take $h : = (1, 0, \cdots, 0, 1) \in \ker{\varphi}$, we see $\ord(h) = n^2$, and $h \not \in \ker{\phi_{\cL}}$. Let $G:= (h) \subset \ker{\varphi}$, and let $X:= Y/G$, denote $\pi: Y \to X$. 

We claim that $\bP E$ descends to a projective bundle $P \in H^1(X, \PGL_n)$. Indeed, since we can assign the descent data $\gamma_{th}: \bP t_{th}^* E \cong \bP E$ to be the unique (since $E$ is simple) isomorphism corresponding to $t_{th}^*E \cong E \otimes N_h^{\otimes t}$ (such $N_h$ exists since $E$ is semi-homogeneous.) To check that this is a descent data, since $G$ is cyclic, we just need to check $\gamma_{n^2h} = \id$. We see $\gamma_{n^2h}$ correspond to $t_{n^2h} E \cong E \otimes N_h^{\otimes n^2}$, so it suffices to show that $N_h^{\otimes n^2} \cong \cO_Y$.

We see that $N_h^{\otimes n} \cong t_h^*{\det{E}} \otimes \det{E}^{-1} = \phi_{\cL}(h)$, so $N_h^{\otimes n^2} \cong [n] \circ \phi_{\cL}(h) = 0 \in Y^t$ as desired. We also see that $[\det{E}]$ does not descend to $X$ since  $t_h^*{\det{E}} \otimes \det{E}^{-1} = \phi_{\cL}(h) \neq 0$.

Let $\alpha \in \Br(X)[n]$ be the image of $P \in H^1(X, \PGL_n)$. By construction, $\pi^* \alpha = 0$, we then can apply the diagram in Remark \ref{diagram}, and similar as before, $\partial(\alpha) = [\det{E}]$. By the following lemma, we see $\partial(\alpha) = [\det{E}] \neq 0 \in \coker{\pi_1}$.

\begin{lemma}
    We have that $[\det{E}] \not \in n\NS(Y) + \pi^* \NS(X)$.
\end{lemma}
\begin{proof}
    We first show that $\pi^* \NS(X) \subset p_i\NS(Y)$ for some $p_i$ a prime factor of $n$. Since $\ker \pi \cong \bZ/{n^2 \bZ}$, there is some $q: X \to Y$ such that $\pi \circ q = [n^2]_X$, so we have $\pi^* \circ q^* = n^4$. This shows that $l = \pi^*{(1)}$ divides $n^4$. Since $[\det E]$ does not descend to $X$, we know $l \neq 1$, so we must have $p_i | n$ for some $p_i$.

    So it suffices to show that $[\det E] \not \in p_i \NS(Y)$. Since $\det E$ does not descend, we see $[\det E] \neq 0 \in \NS(Y)$, so we now show $[\det E] \not \in p_i \NS(Y) - 0$. This is true because any line bundle $\cM \cong M^{\otimes p_i} \in p_i \NS(Y) -0$ has $\chi(\cM) = {p_i}^g \chi(M)$ which is divisible by $p_i^g$, and is nonzero since $\chi(M) \neq 0$, but this contradicts the fact that $\chi(\det E) = n^{g-1} = p_1^{g-1} \cdots p_m^{g-1}$. 
\end{proof}

This key construction gives us the following example where
\begin{itemize}
    \item  $\per(\alpha) = \ind(\alpha) = \Ind_{SH}(\alpha) < \per(\alpha)^g;$
    \item  minimal trivializing isogeny for $\cX$ with degree $> \Ind_{SH}(\alpha)$. 
\end{itemize}

\begin{example}
\label{smart construction}
Let $\alpha$ be the Brauer class we constructed above, and let $n=p$ be a prime number. So $\alpha$ is an example satisfying $\Ind_{SH}(\alpha) = p$ since it admits a rank $p$ semi-homogeneous vector bundle (that corresponds to $P$). But $\pi$ does not trivialize any of its $\mu_p$-lifts since $\partial(\alpha) \neq 0$. 

$\pi$ is an example of a minimal trivializing isogeny of $\cX$ that has degree greater than $\Ind_{SH}(\alpha)$. Indeed, if $\pi$ is not minimal, then it factors through some minimal trivializing isogeny of $\cX$ (i.e. the $\bG_m$-gerbe over $X$ that corresponds to $\alpha$) which then must have degree $p$. By Proposition \ref{trivializing lifts} such a trivializing isogeny would trivialize some of the $\mu_p$-lifts of $\cX$, which is a contradiction.

We also see that $\per(\alpha) = \ind(\alpha) = \Ind_{SH}(\alpha) = p$, and $\per(\alpha)^g > p$ when $g > 1$.
\end{example}

\begin{remark}
    This also gives a concrete construction of the following: given any prime $p$, there exists a pair $(X, \alpha)$ where $X$ is an abelian variety with dimension $g \geq 2$ and $\alpha \in \Br(X)$ with $\per(\alpha) = p$ such that $\ind(\alpha) = \per(\alpha)$.
\end{remark}

Now we prove the existence of $A$ and $E$ as promised at the beginning of the example.

\begin{lemma}
\label{existance of A}
\begin{enumerate}
    \item There exists an abelian variety $A$ of dimension $g$ such that $\NS(A) \cong \bZ$ with generator $H$ having polarization type $(1, 1, , \cdots, 1, n)$ (the first $g-1$ components are 1.).
    \item Let $F: D(A) \cong D(A^t)$ be the Fourier-Mukai equivalence given by the Poincare line bundle on $A \times A^t$, we have that $E:= F(H)$ has $\rk(E) = n$, $|\chi(E)| =1$.
\end{enumerate}
\end{lemma}

\begin{proof}
    It suffices to prove this when $k = \bC$. 

    Proof of (a): By \cite[Theorem 8.2.6 and discussion above it]{lange2013complex}, we see that the moduli space $M_D$ of abelian varieties of type $D = (1, 1, \cdots, 1, n)$ is nonempty. We can then take $A$ to be a very general point on $M_D$.

    Proof of (b): Denote $Y:= A^t$. Let $\Phi(E)$ (and $\Phi(H)$) be the stablizer of $[E]$ ($[H]$) in $Y \times Y^t$ ($A \times A^t$) under the action $\Auteq^0(D(Y)) \cong Y \times Y^t$ on $D(Y)$ ($A \times A^t$ on $D(A)$), this is the same as the definition in \cite[Section 3]{mukai1978semi}. By \cite[Proposition 7.1]{mukai1978semi}, $\rk(E)^2 = | \Phi(E) \cap (0 \times Y^t)|$ and $\rk(H)^2 = | \Phi(H) \cap (0 \times A^t)|$. By \cite[Corollary 7.6]{mukai1978semi}, $\chi(E)^2 = | \Phi(E) \cap (Y \times 0)|$ and $\chi(H)^2 = | \Phi(H) \cap (A \times 0)|$.

    Now we see that $F$ induces isomorphism $\kappa: A \times A^t \cong Y \times Y^t$ such that 
    \begin{itemize}
        \item $\kappa(0 \times A^t)$ is the image of the stabilizer of skyscraper sheaves (for the natural $A \times A^t$ action), which is the stabilizer of degree 0 line bundles on $Y$. So $\kappa(0 \times A^t) = Y \times Y^t$;
        \item $\kappa(A \times 0) = 0 \times Y^t$ for the same reason;
        \item $\kappa(\Phi(H)) = \Phi(E)$.
    \end{itemize}

    So we have that $\rk(E) = \chi(H) = n$ and $\chi(E) =  \pm\rk(H) = \pm 1$.

\end{proof}

\subsection{More Examples}

By Example \ref{smart construction}, we see that $\Ind_{SH}(\alpha)$ can be as small as $\Ind(\alpha)$, and in that example, we have $\per(\alpha) = \ind(\alpha) = \Ind_{SH}(\alpha) < \per(\alpha)^g$. In this subsection, we give examples that illustrate some more possible (in)equality relations between the quantities $\per(\alpha)$, $\Ind(\alpha)$, $\Ind_{SH}(\alpha)$, and $\per(\alpha)^g$.

\begin{example}[$\per(\alpha) < \ind(\alpha) < \Ind_{SH}(\alpha) = \per(\alpha)^g$]
\label{counterexample}

    This is \cite[Example 23.4]{hotchkiss2024period}.

    Let $(X, H)$ be a principally polarized abelian variety with $\dim X =3$ and $\NS(X) \cong \bZ$, let $n=2$. We may write \[H = x_1 \wedge y_1 + x_2 \wedge y_2 + x_3 \wedge y_3\] for some $x_i, y_i \in H^1(X, \mu_n) = \Hom(X[n], \mu_n)$.

    Let $b \in H^2(X, \mu_n)$ be the element \[b = x_1 \wedge (y_1 + y_2) + x_2 \wedge (y_1 + y_3) + x_3 \wedge (y_1 + y_2 + y_3).\] Let $\alpha$ be the image of $b$ in $\Br(X)[2]$.

    We see that the lifts of $\alpha$ in $H^2(X, \mu_n)$ are $b$ and $b+H$ since $\NS(X)/{n\NS(X)} \cong \bZ/{2 \bZ}$. We can calculate that
    \[b+H = x_1 \wedge y_2 + x_2 \wedge (y_1 + y_2 + y_3) + x_3 \wedge (y_1+ y_2).\]

    By Lemma \ref{ann(f)}, $\Ann(b) = \Ann(b+H) = 2^3$, so by Corollary \ref{Ann(alpha)}, $\Ind_{SH}(\alpha) = 2^3 = \per(\alpha)^g$. Since $\alpha \neq 0$ (otherwise $\Ind_{SH}(\alpha) = 1$,) we have $\per(\alpha) = 2$, and by \cite[Theorem 17.1]{hotchkiss2024period}, $\ind(\alpha) \leq 2^2$.
\end{example}

A variation of the example above gives the relation $\per(\alpha) < \ind(\alpha) \leq \Ind_{SH}(\alpha) < \per(\alpha)^g$.

\begin{example}[$\per(\alpha) < \ind(\alpha) \leq \Ind_{SH}(\alpha) < \per(\alpha)^g$]

    Let $(X, H)$ be an abelian variety with $\dim X =3$, $\NS(X) \cong \bZ$, and $\chi(H) = 2$ (that is, $\deg {\phi_{H}}= 2$, and equivalently, the polarization type of $H$ is (1,1,2)), let $n=2$. We may write \[H = x_1 \wedge y_1 + x_2 \wedge y_2 + 2x_3 \wedge y_3 = x_1 \wedge y_1 + x_2 \wedge y_2\] for some $x_i, y_i \in H^1(X, \mu_n) = \Hom(X[n], \mu_n)$.

    Let $b \in H^2(X, \mu_n)$ be the element \[b = x_1 \wedge (y_1 + y_2) + x_2 \wedge (y_1 + y_3) + x_3 \wedge (y_1 + y_2 + y_3).\] Let $\alpha$ be the image of $b$ in $\Br(X)[2]$.

    We see that the lifts of $\alpha$ in $H^2(X, \mu_n)$ are $b$ and $b+H$ since $\NS(X)/{n\NS(X)} \cong \bZ/{2 \bZ}$. We can calculate that
    \[b+H = x_1 \wedge y_2 + x_2 \wedge (y_1 + y_2 + y_3) + x_3 \wedge (y_1+ y_2 + y_3) = x_1 \wedge x_2 + (x_2 + x_3) \wedge (y_1 + y_2 + y_3).\]

    By Lemma \ref{ann(f)}, $\Ann(b) = 2^2, \Ann(b+H) = 2^3$, so by Corollary \ref{Ann(alpha)}, $\Ind_{SH}(\alpha) = 2^2 < \per(\alpha)^g$. Since $\alpha \neq 0$, we have $\per(\alpha) = 2$, and by \cite[Theorem 17.1]{hotchkiss2024period}, $\ind(\alpha) \leq 2^2$.
\end{example}

\section{Product of Elliptic Curves}

For simplicity, we choose isomorphisms $\mu_n \cong \bZ/{n \bZ}$. Note that we can do this for every integer $n$ since $\Char{k} = 0$.

In this section, we study the special case when $X = E_1 \times \cdots E_g$ where $E_i$ are elliptic curves. We may choose the basis $x_i, y_i \in H^1(X, \bZ/ n \bZ)$ such that the ample generator of $\NS(E_i) \subset \NS(X)$ can be written as \[H_i = x_i \wedge y_i.\]

In the case where $E_i's$ are not isogeneous, $\NS(X) \cong \oplus_i\bZ \cdot H_i$. In general, $\oplus_i \bZ \cdot H_i \xhookrightarrow{} \NS(X)$ and other factors correspond to the graphs of isogenies $E_i \to E_j$.

Fix the basis of $H^2(X, \bZ /n \bZ)$ to be the basis we chose above. We then have \[H_i = x_i \wedge y_i \in H^2(X, \bZ/{n\bZ})\] and hence $a_i M_{H_i}$ is the matrix
\[
\mathrm{diag}\!\left(
\begin{pmatrix}0 & 0 \\ 0 & 0\end{pmatrix},
\ldots,
\begin{pmatrix}0 & a_i \\ -a_i & 0\end{pmatrix},
\ldots,
\begin{pmatrix}0 & 0 \\ 0 & 0\end{pmatrix}
\right)
.\]

\begin{lemma}
\label{Pf}
    Let $n$ be an integer, given any $f \in H^2(X, \bZ/{n \bZ})$, $f$ corresponds to an off-diagonal skew symmetric matrix $M_f: X[n] \to X^t[n]$. If $\Pf(M_f) = 0 \in \bZ/{n\bZ}$, we have $h(f) \leq n^{g-1}$.
\end{lemma}

\begin{proof}
    By changing basis, $M_f$ can be written as the block matrix 
    \[
\mathrm{diag}\!\left(
\begin{pmatrix}0 & d_1 \\ -d_1 & 0\end{pmatrix},
\begin{pmatrix}0 & d_2 \\ -d_2 & 0\end{pmatrix},
\ldots,
\begin{pmatrix}0 & d_g \\ -d_g & 0\end{pmatrix}
\right)
.\]

So $\im(M_f) = (\prod_{i=1}^g\frac{n}{\gcd(d_i, n)})^2$ which divides $(\frac{n^{g}}{\gcd(d_1 \cdots d_g, n)})^2$. Since $\Pf(M_f) = d_1 \cdots d_g = 0$, we have $n | d_1 \cdots d_g$, which implies that \[h(f)  | n^{g-1}\] since $h(f) = \sqrt{|\im{M_f}|}$.
\end{proof}

\begin{remark}
    This proof did not use the assumption of $X$ being a product of elliptic curves. That is, it works for any abelian variety $X$.
\end{remark}

\begin{lemma}
\label{induction of Pf}
    Let $M$ be a $2g \times 2g$ skew symmetric matrix. Then there exists $(a_1, \cdots, a_g) \in (\bZ/{n\bZ})^g$ such that $\Pf(N) = 0$ where $N = M + \sum_{i=1}^g a_i M_{H_i}$
\end{lemma}

\begin{proof}
    We induct on $g$. Denote the $ij$-th entry of $M$ by $c_{ij}$.

     Base case: when $g = 1$, then 
     \[N = \begin{pmatrix}
         0 & c_{12}+a_1 \\
         -c_{12}-a_1 & 0
     \end{pmatrix}\]
    So letting $a_1 = -c_{12}$ proves the claim for the rank 2 ($g=1$) case.

     Now we suppose the claim is true for rank $2g$ and we now show it is true for rank $2g+2$. By the inductive hypothesis, we see that by adding some $\sum_{i=1}^g a_i M_{H_i}$ to any skew symmetric matrix of rank $2g$, we can get a resulting matrix with 0 Pfaffian.

     Induction Step:
    Choose $a_{g+1} = -c_{(2g+1)(2g+2)}+1$, we may write \[N = 
    \begin{pmatrix}
        M' & B\\
        -B^T & J
    \end{pmatrix}.
    \]  Here $M' = A + \sum_{i=1}^{g} a_i M_{H_i}$ where $A$ is an off-diagonal skew symmetric $2g \times 2g$ matrix ($A$ is the upper left $2g \times 2g$ block of $M$), $B$ is a $2g \times 2$ matrix independent of the $a_i$'s, and $J = 
    \begin{pmatrix}
        0 & 1\\
        -1 & 0
    \end{pmatrix}$. 

    So \[\Pf(N) = \Pf(J) \Pf(M' + B J^{-1} B^T) = \Pf(J) \Pf(A + B J^{-1} B^T +\sum_{i=1}^g a_i M_{H_i}).\]

    By inductive hypothesis, we see that we may choose $(a_1, \cdots a_g)$ such that $$\Pf(A + B J^{-1} B^T +\sum_{i=1}^g a_i M_{H_i}) = 0,$$ which then implies that $\Pf(N) = 0$.
\end{proof}

\begin{lemma}
\label{orbit}
    Given any $f \in H^2(X, \mu_n) \cong H^2(X, \bZ/{n\bZ})$, there exists $(a_1, \cdots, a_g) \in (\bZ/{n\bZ})^g$ such that $h(f')| n^{g-1}$ where $f' = f + \sum_{i=1}^g a_i H_i$.
\end{lemma}

\begin{proof}
    By Lemma \ref{Pf}, it suffices to show that there exists $(a_1, \cdots, a_g) \in (\bZ/{n\bZ})^g$ such that $\Pf(M_{f'}) = 0$. We then conclude by Lemma \ref{induction of Pf}.
\end{proof}

\begin{theorem}
\label{Elliptic}
    Let $ X = E_1 \times \cdots E_g$. Given any $\alpha \in \Br(X)$, let $n = \per(\alpha)$. Then, there exists a lift $f \in H^2(X, \mu_n) \cong H^2(X, \bZ/{n\bZ})$ of $\alpha$ such that $\Ann(f) | n^{g-1}$. In particular, $\ind(\alpha) | n^{g-1}$.
\end{theorem}

\begin{proof}
    Let $f'$ be an arbitrary lift of $\alpha$. By Lemma \ref{orbit}, we see that there is always $f \in f' + \NS(X)/{n \NS(X)}$ such that $h(f) | n^{g-1}$. That is, there is always a lift $f$ such that $\Ann(f) = h(f) | n^{g-1}$. By Corollary \ref{divisibility}, $\ind(\alpha) | n^{g-1}$.
\end{proof}

\begin{remark}
    The proofs in this section work for more general base field $k$ when $\Char(k)  \nmid \per(\alpha)$.
\end{remark}





    




\printbibliography
\end{document}